\documentclass[twoside,a4paper,reqno,11pt]{amsart} 
\usepackage[top=28mm,right=28mm,bottom=28mm,left=28mm]{geometry}
\usepackage{microtype}
\usepackage{amsfonts, amsmath, amssymb, amsbsy, mathrsfs, bm, latexsym, stmaryrd, hyperref, array, enumitem, textcomp, url}
\usepackage[table]{xcolor}

\renewcommand{\a}{\alpha}
\renewcommand{\b}{\beta}

\newcommand{\e}{\epsilon}

\newcommand{\s}{\sigma}
\renewcommand{\O}{\Omega}

\newcommand{\la}{\langle}
\newcommand{\ra}{\rangle}
\newcommand{\leqs}{\leqslant}
\newcommand{\geqs}{\geqslant}
\newcommand{\what}{\widehat}

\newcommand{\vs}{\vspace{3mm}}

\makeatletter
\newcommand{\imod}[1]{\allowbreak\mkern4mu({\operator@font mod}\,\,#1)}
\makeatother

\theoremstyle{plain}
\newtheorem{theorem}{Theorem}

\newtheorem{thm}{Theorem}[section] 
\newtheorem{lem}[thm]{Lemma}

\newtheorem*{theorem*}{Theorem} 
\newtheorem*{conj*}{Conjecture}

\theoremstyle{definition}
\newtheorem{rem}[thm]{Remark}

\newtheorem{remk}{Remark}

\begin{document}

\title[On soluble subgroups of sporadic groups]{On soluble subgroups of sporadic groups}

\author{Timothy C. Burness}
\thanks{The author thanks Thomas Breuer and Eamonn O'Brien for useful discussions on the content of this paper. He also thanks an anonymous referee for their careful reading of an earlier draft and for many helpful suggestions regarding the computations in Section 2.}
\address{T.C. Burness, School of Mathematics, University of Bristol, Bristol BS8 1UG, UK}
\email{t.burness@bristol.ac.uk}

\date{\today} 

\begin{abstract}
Let $G$ be an almost simple sporadic group and let $H$ be a soluble subgroup of $G$. In this paper we prove that there exists $x,y \in G$ such that $H \cap H^x \cap H^y=1$, which is equivalent to the bound $b(G,H) \leqs 3$ with respect to the base size for the natural action of $G$ on the set of cosets of $H$. This bound is best possible. In this setting, our main result establishes a strong form of a more general conjecture of Vdovin on the intersection of conjugate soluble subgroups of finite groups. The proof uses a combination of computational and probabilistic methods.
\end{abstract}

\maketitle

\section{Introduction}\label{s:intro}

Let $G \leqs {\rm Sym}(\Omega)$ be a transitive permutation group on a finite set $\O$ with point stabiliser $H$. A subset of $\O$ is a \emph{base} for $G$ if its pointwise stabiliser in $G$ is trivial; the minimal size of a base is called the \emph{base size} of $G$, denoted $b(G,H)$. Determining the base size of a permutation group is a fundamental problem and the study of bases has a long history in permutation group theory, stretching all the way back to the nineteenth century. We refer the reader to the survey articles \cite{BC,LSh3} and \cite[Section 5]{Bur181} for more background on bases and their applications and connections to other areas of algebra and combinatorics.

For more than a century, there has been a focus on studying bases for primitive permutation groups, in which case a point stabiliser $H$ is a maximal subgroup of $G$. In more recent years, further interest in this setting stems from highly influential conjectures of Cameron, Kantor and Pyber from the 1990s, which have attracted significant attention from various authors. Here Cameron's base size conjecture is the most relevant to the theme of this paper. In order to state the  conjecture, let us first recall that $G$ is \emph{almost simple} if $G_0 \leqs G \leqs {\rm Aut}(G_0)$ for some nonabelian finite simple group $G_0$, which is the socle of $G$. An almost simple primitive group $G \leqs {\rm Sym}(\O)$ is said to be \emph{standard} if $G_0 = A_m$ is an alternating group and $\O$ is a set of subsets or partitions of $\{1, \ldots, m\}$, or if $G_0$ is a classical group and $\O$ is a set of subspaces (or pairs of subspaces) of the natural module for $G_0$ (otherwise, $G$ is \emph{non-standard}). In \cite{CK}, Cameron and Kantor conjectured that there exists an absolute constant $c$ such that 
$b(G,H) \leqs c$ for every non-standard group $G$ (in contrast, if $G$ is standard then typically $b(G,H)$ can be arbitrarily large). This was refined by Cameron \cite[p.122]{CamPG}, who conjectured that $b(G,H) \leqs 7$, with equality if and only if $G$ is the Mathieu group ${\rm M}_{24}$ in its natural action on $24$ points. 

The original form of the Cameron-Kantor conjecture was proved by Liebeck and Shalev \cite{LSh2} using probabilistic methods and fixed point ratio estimates. By applying similar techniques, Cameron's refined conjecture was established in the sequence of papers \cite{Bur7,BGS,BLS,BOW}. In particular, we have $b(G,H) \leqs 7$ for every primitive almost simple sporadic group $G$ with point stabiliser $H$, with equality if and only if $G = {\rm M}_{24}$ and $H = {\rm M}_{23}$. In fact, in this setting the exact base size of $G$ is determined in \cite{BOW} (combined with \cite{NNOW} for two special cases involving the Baby Monster). 

In this paper, we extend the work on bases for sporadic groups in \cite{BOW}. Let $G \leqs {\rm Sym}(\O)$ be a transitive almost simple sporadic group with soluble point stabiliser $H$. If $G$ is primitive, then the main theorem of \cite{BOW} implies that $b(G,H) \leqs 3$ and one of our main aims is to show that the same bound holds when $G$ is transitive.

\begin{theorem}\label{t:main1}
Let $G \leqs {\rm Sym}(\O)$ be a finite almost simple transitive permutation group with socle $G_0$ and point stabiliser $H$. If $G_0$ is a sporadic group and $H$ is soluble, then $b(G,H) \leqs 3$. 
\end{theorem}

Notice that Theorem \ref{t:main1} is an immediate corollary of the following more general result. Given a finite group $G$, let $\mathcal{S}(G)$ be the set of core-free soluble subgroups of $G$ and set 
\begin{equation}\label{e:sG}
s(G) = \max\{b(G,H) \,:\, H \in \mathcal{S}(G)\}.
\end{equation}
In the following statement, we use $\mathbb{B}$ to denote the Baby Monster sporadic group.

\begin{theorem}\label{t:main2}
Let $G$ be a finite almost simple sporadic group with socle $G_0$.
\begin{itemize}\addtolength{\itemsep}{0.2\baselineskip}
\item[{\rm (i)}] If $G \ne \mathbb{B}$, then 
\[
s(G) = \left\{ \begin{array}{ll}
3 & \mbox{if $G_0 = {\rm M}_{11}$, ${\rm M}_{12}$, ${\rm M}_{22}$, ${\rm M}_{23}$, ${\rm M}_{24}$, ${\rm J}_{2}$, ${\rm Co}_{2}$, ${\rm Fi}_{22}$ or ${\rm Fi}_{23}$} \\
2 & \mbox{otherwise.}
\end{array}\right.
\]
\item[{\rm (ii)}] If $G = \mathbb{B}$, then $s(G) \leqs 3$.
\end{itemize}
\end{theorem}

\begin{remk}\label{r:1}
Some comments on the statement of Theorem \ref{t:main2} are in order.
\begin{itemize}\addtolength{\itemsep}{0.2\baselineskip}
\item[{\rm (a)}] First observe that Theorem \ref{t:main2} shows that the upper bound in Theorem \ref{t:main1} is best possible. For each group with $s(G)=3$ in part (i) of Theorem \ref{t:main2}, we refer the reader to Table \ref{tab:sol} for an example of a soluble subgroup $H$ with $b(G,H) = 3$ (in the table, $c \in \{1,2\}$). Further information on the possibilities for $H$ is given in 
Remarks \ref{r:sol1} and \ref{r:sol2}.
\item[{\rm (b)}] The analysis of the Baby Monster $G = \mathbb{B}$ presents several challenges, both theoretically and computationally, and it will be handled separately in Section \ref{s:proof3}. We will establish the bound $s(G) \leqs 3$, but we do not know the precise value of $s(G)$ in this case. Here the analysis is complicated by the existence of a maximal subgroup $L = 2.{}^2E_6(2){:}2$ of $G$ with $b(G,L)=4$ (see \cite[Theorem 1]{BOW}); if $H$ is contained in $L$, then some work is needed to verify the bound $b(G,H) \leqs 3$. See Remark \ref{r:baby} for further comments.
\item[{\rm (c)}] For the Monster $G = \mathbb{M}$, it is worth noting that we will prove that $b(G,H) \leqs 3$ for \emph{every} proper subgroup $H$ of $G$, with equality if and only if $H$ is the involution centraliser $2.\mathbb{B}$ (see Theorem \ref{t:monster}).
\item[{\rm (d)}] We have excluded the almost simple groups with socle $G_0 = {}^2F_4(2)'$ in Theorem \ref{t:main2}. For completeness let us record that $s(G) = 3$, noting that $G = G_0{:}c$ has soluble maximal subgroups $H = 2^2.[2^{7+c}].S_3$ and $2.[2^{7+c}].5.4$ with $b(G,H) = 3$. This is easy to verify using the same computational methods we employ in the proof of Theorem \ref{t:all} in Section \ref{s:proof1}.
\end{itemize}
\end{remk}

\begin{table}
\[
\begin{array}{lll} \hline
G & H & \mbox{Comments} \\ \hline
{\rm M}_{11} & 3^2{:}SD_{16} & \mbox{maximal in $G$} \\
{\rm M}_{12}{:}c & 2^{1+4}.S_3 & \mbox{maximal in $G_0$} \\
{\rm M}_{22}{:}c & 2^4{:}S_4 & H < 2^{4}{:}S_5 < G_0 \\
{\rm M}_{23} & 2^4{:}(A_4 \times 3).2 & H < 2^4{:}A_7 < G \\
{\rm M}_{24} &  2^6{:}3.(S_3 \wr S_2) & H < 2^6{:}3.S_6 < G \\
{\rm J}_2{:}c & 2^{2+4}{:}(3 \times S_3) & \mbox{maximal in $G_0$} \\
{\rm Co}_2 & 2^{4+10}.(S_4 \times S_3) & H < 2^{4+10}.(S_5 \times S_3) < G \\ 
{\rm Fi}_{22}{:}c & 3^{1+6}{:}2^{3+4}{:}3^2{:}2.c & \mbox{maximal in $G$} \\ 
{\rm Fi}_{23} & 3^{1+8}.2^{1+6}.3^{1+2}.2S_4 & \mbox{maximal in $G$} \\ \hline
\end{array}
\]
\caption{Some examples with $b(G,H) = 3$, $H$ soluble}
\label{tab:sol}
\end{table}

Further motivation for considering finite permutation groups with soluble point stabilisers arises from a conjecture of Vdovin. Let $G$ be a finite group, let $H$ be a soluble subgroup of $G$ and assume $G$ has no nontrivial soluble normal subgroups. In \cite[Problem 17.41(b)]{Kou}, Vdovin conjectures that there exist four elements $x_1, \ldots, x_4 \in G$ such that 
\[
H \cap H^{x_1} \cap H^{x_2} \cap H^{x_3} \cap H^{x_4} = 1.
\]
In other words, if we view $G$ as a transitive permutation group on the set of cosets of $H$, then the conjecture asserts that $b(G,H) \leqs 5$ (in \cite{BGP}, Babai, Goodman and Pyber conjectured the weaker bound $b(G,H) \leqs 7$). In a recent paper \cite{B2020}, the author has established the bound $b(G,H) \leqs 5$ for every finite primitive permutation group $G$ with soluble point stabiliser $H$, which proves Vdovin's conjecture in the case where $H$ is a maximal subgroup of $G$. Let us also observe that the bound $b(G,H) \leqs 5$ is best possible; for example, if $G = S_8$ and $H = S_4 \wr S_2$ then $b(G,H) = 5$.

The general form of Vdovin's conjecture remains open, but there has been some important progress. In \cite{Vdovin2}, Vdovin reduces the conjecture to a problem concerning almost simple groups. More precisely, it suffices to show that if $G \leqs {\rm Sym}(\O)$ is an almost simple transitive group with socle $G_0$ and soluble point stabiliser $H$, then $G$ has at least $5$ regular orbits with respect to its natural action on the cartesian product $\O^5$ (note that $b(G,H) \leqs 5$ if and only if $G$ has at least one regular orbit on this set). In \cite{Bay2}, Baykalov proves that $G$ has at least $5$ regular orbits when $G_0$ is an alternating group and there is ongoing work and partial results for groups of Lie type. For example, the main results in \cite{Bay0} handle the almost simple classical groups with socle a linear, unitary or symplectic group, under the additional assumption that $G$ does not contain graph or graph-field automorphisms in the linear and symplectic cases. In addition, Vdovin \cite{Vdovin3} has established the desired result in the special case where $G$ is an exceptional group of Lie type and $H$ is a soluble Hall subgroup of $G$. 

As noted in \cite[Lemma 3]{Bay2}, if $b(G,H) \leqs 4$ then $G$ has at least $5$ regular orbits on $\O^5$, so Theorem \ref{t:main1} establishes the desired condition for all almost simple sporadic groups (in a strong form), bringing us a step closer to a proof of Vdovin's conjecture. It also extends a special case of a theorem of Vdovin and Zenkov \cite[Theorem 2]{VZ}, which states that $b(G,H) \leqs 5$ when $G$ is an almost simple sporadic group and $H$ is a soluble Hall subgroup of $G$.

\begin{remk}
We note that a similar problem has recently been studied by Zenkov \cite{Zenkov} with regard to nilpotent subgroups of almost simple sporadic groups. His main result states that if $H$ and $K$ are nilpotent subgroups of such a group $G$, then there exists $x \in G$ such that $H \cap K^x = 1$. In particular, $b(G,H)=2$ with respect to the action of $G$ on the set of cosets of $H$.
\end{remk}

\begin{remk}
It is also worth highlighting some related results of Breuer, which are documented in Chapter 6 of the manual for the \textsf{GAP} Character Table Library \cite{GAPCTL}. For each almost simple sporadic group $G$, a computational approach is used to calculate the maximal order $m$ of a soluble subgroup of $G$. Moreover, Breuer determines the conjugacy classes of soluble subgroups of order $m$, as well as their maximal overgroups in $G$. In particular, we observe that $|H| < |G|^{2/3}$ for every soluble subgroup $H$, which is a necessary condition for the bound $b(G,H) \leqs 3$ we establish in Theorem \ref{t:main1} (see \eqref{e:easy}).
\end{remk}

Our proof of Theorem \ref{t:main2} naturally falls into three cases. The main result of Section \ref{s:proof1} handles the groups $G \ne \mathbb{M}, \mathbb{B}$, noting that the Monster and Baby Monster require special attention and they will be the focus of  Sections \ref{s:proof2} and \ref{s:proof3}, respectively. For $G \ne \mathbb{M},\mathbb{B}$ we use computational methods (working with {\sc Magma} \cite{magma} and \textsf{GAP} \cite{GAP}) to provide an essentially uniform approach to the problem and we refer the reader to Section \ref{s:proof1} for an overview of the main techniques. Our approach for $\mathbb{M}$ and $\mathbb{B}$ relies on a powerful probabilistic method for studying bases, which was introduced by Liebeck and Shalev \cite{LSh2} in their proof of the Cameron-Kantor conjecture. The probabilistic set-up will be explained in Section \ref{s:proof2}.

\vs

Finally, let us comment on the notation we use in this paper, which is all fairly standard. Let $G$ be a finite group and let $n$ be a positive integer. We will write $C_n$, or just $n$, for a cyclic group of order $n$ and $G^n$ will denote the direct product of $n$ copies of $G$. An unspecified extension of $G$ by a group $H$ will be denoted by $G.H$; if the extension splits then we write $G{:}H$. We use $[n]$ for an unspecified soluble group of order $n$. If $X$ is a subset of $G$, then $i_n(X)$ is the number of elements of order $n$ in $X$. We adopt the standard notation for simple groups of Lie type from \cite{KL}. In particular, ${\rm L}_{n}^{\e}(q)$ denotes ${\rm PSL}_{n}(q)$ (when $\e=+$) and ${\rm PSU}_{n}(q)$ (when $\e=-$). We write ${\rm P\O}_{n}^{\e}(q)$ for the simple orthogonal groups and it is worth noting that this differs from the notation used in the \textsc{Atlas} \cite{Atlas}.

\section{Proof of Theorem \ref{t:main2}, $G \ne \mathbb{M}, \mathbb{B}$}\label{s:proof1}

Let $G \leqs {\rm Sym}(\O)$ be an almost simple transitive group with sporadic socle $G_0$ and soluble point stabiliser $H$. Note that if $n = |\O| = |G:H|$ and $B \subseteq \O$ is a base for $G$, then $|G| \leqs n^{|B|}$. In particular, if $b(G,H)$ denotes the base size of $G$, then we have
\begin{equation}\label{e:easy}
b(G,H) \geqs \log_{n}|G|.
\end{equation}
Let us also recall the definition of $s(G)$ in Theorem \ref{t:main2} (see \eqref{e:sG}). 

As noted in Section \ref{s:intro}, the Monster $\mathbb{M}$ and Baby Monster $\mathbb{B}$ require special attention and they will be handled separately in Sections \ref{s:proof2} and \ref{s:proof3}. Here we will prove the following result, which establishes Theorem \ref{t:main2} for the remaining sporadic groups.

\begin{thm}\label{t:all}
Let $G \ne \mathbb{M},\mathbb{B}$ be a finite almost simple sporadic group with socle $G_0$.
Then
\[
s(G) = \left\{ \begin{array}{ll}
3 & \mbox{if $G_0 = {\rm M}_{11}$, ${\rm M}_{12}$, ${\rm M}_{22}$, ${\rm M}_{23}$, ${\rm M}_{24}$, ${\rm J}_{2}$, ${\rm Co}_{2}$, ${\rm Fi}_{22}$ or ${\rm Fi}_{23}$} \\
2 & \mbox{otherwise.}
\end{array}\right.
\]
\end{thm}

Our proof relies entirely on computational methods, working primarily with {\sc Magma} \cite{magma} (version V2.26-6), together with some additional input provided by \textsf{GAP} \cite{GAP} (version 4.11.1). Our approach is essentially uniform, although there are a few differences between cases that we will highlight below. Here we provide a brief summary of the main steps. 

Let $G$ be an almost simple sporadic group as in Theorem \ref{t:all} and suppose we seek to establish the bound $s(G) \leqs c$. In the first step, we use {\sc Magma} to work with a faithful permutation or matrix representation of $G$. More precisely, if $G \ne {\rm Th}$, ${\rm J}_4$ or ${\rm Ly}$ then we use the {\sc Magma} function \texttt{AutomorphismGroupSimpleGroup} (denoted \texttt{AGSG} for short) to construct ${\rm Aut}(G_0)$ (and subsequently the socle $G_0$) as a permutation group. And we use
the function \texttt{MatrixGroup} to construct the three remaining groups ${\rm Th}$, 
${\rm J}_4$ and ${\rm Ly}$ as matrix groups of dimension $248$, $112$ and $111$ over the fields $\mathbb{F}_2$, $\mathbb{F}_2$ and $\mathbb{F}_5$, respectively.

Next let $H$ be a soluble subgroup of $G$ and embed $H$ in a maximal subgroup $M$ of $G$. The base size $b(G,M)$ is computed in \cite{BOW} and so we may assume $H<M$ is a proper subgroup and $b(G,M) \geqs 3$. In turn, we can embed $H$ in a maximal subgroup $K$ of $M$ and then either 
\begin{itemize}\addtolength{\itemsep}{0.2\baselineskip}
\item[{\rm (a)}] use random search to identify elements $g_1, \ldots, g_c$ in $G$ such that $\bigcap_iK^{g_i} = 1$, which implies that $b(G,H) \leqs c$; or
\item[{\rm (b)}] we can embed $H$ in a maximal subgroup of $K$ and repeat. 
\end{itemize}
This is our basic approach and in most cases we only need to descend to subgroups in the third or fourth layer of the subgroup lattice in order to establish the existence of an overgroup $J$ of $H$ with $b(G,J) \leqs c$. 

If this goes through with $c=2$, then we conclude that $s(G) = 2$. On the other hand, if we want to show that $s(G) = 3$ then we also need to exhibit a soluble subgroup $H$ with $b(G,H) = 3$ and there are several ways to do this. For example, $G$ may have a soluble maximal subgroup with $b(G,H)=3$; these cases can be read off from \cite{BOW}. Otherwise, we may be able to find a soluble subgroup $H$ such that $b(G,H) \leqs 3$ and $\log_{n}|G|>2$, in which case $b(G,H) = 3$ by \eqref{e:easy}. Finally, in a handful of cases we will need to show that $b(G,H) = 3$ for a soluble subgroup $H$ with $\log_n|G| \leqs 2$. Here we need to rule out the existence of a regular orbit of $H$ with respect to the natural action on $G/H$ and we can work effectively with $(H,H)$ double cosets to do this (for example, see  the case $G = {\rm M}_{24}$ in the proof of Theorem \ref{t:all} below).

The main computational challenge in implementing this approach involves constructing representatives of the relevant conjugacy classes of maximal subgroups of $G$ (and also maximal subgroups of these subgroups, and so on,  if we need to go deeper into the subgroup lattice). To do this, we can use the function \texttt{MaximalSubgroups} (\texttt{MS} for short) in the vast majority of cases, working with a faithful permutation representation of $G$. The only exceptions are the groups with socle $G_0 = {\rm Co}_1$ or ${\rm Fi}_{24}'$. In these two special cases, we will use the \textsf{GAP} package \texttt{AtlasRep} \cite{AR} (version 2.1.0), which provides black-box algorithms to construct generators for every maximal subgroup of ${\rm Aut}(G_0)$ that are conveniently presented as words in the standard generators for ${\rm Aut}(G_0)$ provided by \texttt{AGSG}. Finally, if $G$ is one of the groups ${\rm Th}$, ${\rm J}_4$ or ${\rm Ly}$ then we use the function \texttt{LMGMaximalSubgroups} with respect to the relevant matrix representation of $G$ mentioned above (this is part of the \texttt{CompositionTree} {\sc Magma} package for computing with matrix groups; see \cite{Comp}).

\begin{proof}[Proof of Theorem \ref{t:all}]
Our general approach is described above and so we only provide details in a selective number of cases, which are designed to illustrate the full range of techniques we apply.  Specifically, we will assume the socle $G_0$ of $G$ is one of the following:
\[
{\rm (i)} \; {\rm M}_{24}, \; {\rm (ii)} \; {\rm Co}_{1}, \; {\rm (iii)} \; {\rm Fi}_{24}', \; {\rm (iv)} \; {\rm J}_4.
\]
We leave the reader to verify that the desired conclusion holds in the remaining cases. Let $H$ be a soluble subgroup of $G$.

First consider case (i), so $G = {\rm M}_{24}$ and we claim that $s(G)=3$.
Here \texttt{AGSG} returns $G$ as a permutation group in its natural action on $24$ points and \texttt{MS} returns a set of representatives of the $9$ conjugacy classes of maximal subgroups of $G$ (all of which are insoluble). Fix a maximal subgroup $M = 2^6{:}3.S_6$ of $G$ and use \texttt{MS} again to construct the maximal subgroups of $M$. 
For each maximal subgroup $K<M$, we can use random search to find elements $x,y \in G$ such that $K \cap K^x \cap K^y=1$, which implies that $b(G,H) \leqs 3$ if $H$ is contained in $M$. Moreover, we claim that $b(G,K) = 3$ if $K = 2^6{:}3.(S_3 \wr S_2)$. To see this, we need to rule out the existence of a regular orbit of $K$ in its action on $G/K$. An effective way to do this is to show that there are no $(K,K)$ double cosets in $G$ of size $|K|^2$, noting that $|KgK| = |K|^2/|K \cap K^g|$. This is readily checked using the \texttt{DoubleCosetRepresentatives} function in {\sc Magma}.

In view of \cite[Theorem 1]{BOW}, it remains to show that $b(G,H) \leqs 3$ when $H$ is contained in one of the maximal subgroups $M = {\rm M}_{23}$, ${\rm M}_{22}{:}2$ and $2^4{:}A_8$ of $G$. As before, it suffices to show that there exist elements $x,y \in G$ such that $K \cap K^x \cap K^y = 1$ for each maximal subgroup $K$ of $M$. This is easy to verify using random search unless $K = {\rm M}_{22}$, which arises as a maximal subgroup of both ${\rm M}_{23}$ and ${\rm M}_{22}{:}2$ (indeed, we have $b(G,K) = 4$). But if $H$ is contained in $K$, then it is contained in a maximal subgroup $J<K$ and by repeating the process, using a combination of \texttt{MS} and random search, it is easy to check that $b(G,J) \leqs 3$. We conclude that $s(G) = 3$ as required.

Next let $G = {\rm Co}_1$. By \cite{BOW}, we may assume $H<M$, where $M$ is one of the following maximal subgroups of $G$:
\[
{\rm (a)} \; {\rm Co}_2, \;\; {\rm (b)} \; 3.{\rm Suz}{:}2,\;\; {\rm (c)} \; 2^{11}{:}{\rm M}_{24}, \;\; {\rm (d)} \; {\rm Co}_3, \;\; {\rm (e)} \; 2^{1+8}.\O_8^{+}(2), \;\; {\rm (f)} \; {\rm U}_6(2){:}S_3,
\]
\[
{\rm (g)} \; (A_4 \times G_2(4)){:}2,\;\; {\rm (h)} \; 2^{2+12}{:}(A_8 \times S_3),\;\; {\rm (i)} \; 2^{4+12}.(S_3 \times 3.S_6).
\]
Using \texttt{AGSG} we can work with $G$ as a permutation group of degree $98280$ and we can construct each maximal subgroup $M$ in cases (a)--(f) via the command \texttt{MaximalSubgroups(G,"Co1")}. These cases can then be handled as above,  using the function \texttt{MS} to descend inside $M$ as needed. Generators for the remaining subgroups in (g)-(i) are available via the \textsf{GAP} package \texttt{AtlasRep} \cite{AR} and they are presented as words in the standard generators for $G$ provided by \texttt{AGSG}. Working with these generators in {\sc Magma}, we can now proceed as before in order to show that $b(G,H) = 2$ and thus $s(G) = 2$.

Now suppose $G_0 = {\rm Fi}_{24}'$, so $G = G_0$ or $G_0.2$. To begin with, let us assume $G = G_0.2$. Here we first use the function \texttt{AGSG} to construct $G$ as a permutation group of degree $306936$. Next we inspect \cite{BOW} to read off the maximal subgroups $M$ of $G$ with $b(G,M) \geqs 3$ and in each case we use the generators provided by the \textsf{GAP} \texttt{AtlasRep} package to construct $M$ as a subgroup of $G$. We can now complete the analysis in the usual manner, working with \texttt{MS} to construct the maximal subgroups of $M$. For $G = G_0$ we observe that each relevant maximal subgroup $M$ of $G$ is of the form $L \cap G$, where $L$ is a maximal subgroup of $G_0.2$. Therefore, we can construct $L$ as above, intersect with $G$ to obtain $M$ and then continue as before. 

Finally, let us assume $G = {\rm J}_4$. By \cite{BOW} we may assume $H$ is contained in a maximal subgroup $M$ of $G$, where $M$ is either $2^{11}{:}{\rm M}_{24}$, 
$2^{1+12}.3.{\rm M}_{22}{:}2$ or $2^{10}{:}{\rm L}_{5}(2)$. First we apply the \texttt{MatrixGroup} function in {\sc Magma} to construct $G < {\rm GL}_{112}(2)$ and we can then use \texttt{LMGMaximalSubgroups} to construct   each possibility for $M$. We use the same function to descend deeper into the subgroup lattice of $G$ and we can randomly search (in the usual way) for an element $x \in G$ such that $K \cap K^x = 1$ for some overgroup $K$ of $H$. In this way, we deduce that $s(G) = 2$.
\end{proof}

\begin{rem}\label{r:sol1}
Let $G \ne {\rm Fi}_{23}, \mathbb{M}, \mathbb{B}$ be an almost simple sporadic group with $s(G) = 3$ and socle $G_0$. By Theorem \ref{t:all} we have $G_0  \in \mathcal{A}$, where
\[
\mathcal{A} = \{ {\rm M}_{11}, {\rm M}_{12}, {\rm M}_{22}, {\rm M}_{23}, {\rm M}_{24}, {\rm J}_{2}, {\rm Co}_{2}, {\rm Fi}_{22} \}.
\]
In Table \ref{tab:sol} we give an example of a soluble subgroup $H$ with $b(G,H) = 3$ and with some additional work (using the same computational methods from the proof of Theorem \ref{t:all}) it is possible to describe all the relevant soluble subgroups $H$ of $G$. More precisely, in Table \ref{tab:solH} we give the structure of each maximal soluble subgroup $H$ of $G$ with $b(G,H) = 3$, together with the indices $m$ of the proper subgroups $K<H$ with $b(G,K) = 3$ (if no value is recorded in this column, then $b(G,K) = 2$ for every nontrivial proper subgroup $K$ of $H$). As one might expect, the possibilities for $H$ are very restricted. For example, if $G = {\rm M}_{23}$, ${\rm Co}_2$ or ${\rm Fi}_{22}$ then $G$ has a unique conjugacy class of soluble subgroups $H$ with $b(G,H) = 3$.
\end{rem}

\begin{table}
\[
\begin{array}{llcl} \hline
G & H & m & \mbox{Comments} \\ \hline 
{\rm M}_{11} & 3^2{:}SD_{16} & 2 & \mbox{$H$ maximal in $G$} \\
{\rm M}_{12} & 4^2{:}D_{12}, \, 2^{1+4}{:}S_3, \, 3^2{:}2S_4 &  & \mbox{$H$ maximal in $G$} \\
{\rm M}_{12}{:}2 & 3^{1+2}{:}D_8 &  & \mbox{$H$ maximal in $G$} \\
& 4^2{:}D_{12}.2,\, 2^{1+4}{:}D_{12} & 2 & \mbox{$H$ maximal in $G$} \\
{\rm M}_{22} & 2^4{:}S_4 &  & H<2^4{:}S_5,\, H < 2^4{:}A_6 \\
& 2^4{:}3^2.4 & 2 & H<2^4{:}A_6 \\
{\rm M}_{22}{:}2 & 2^5{:}S_4 & 2 & H<2^5{:}S_5 \\
& 2^4{:}(S_4 \times S_2), \, 2^4{:}(S_2 \wr S_3) & 2 & H < 2^4{:}S_6 \\
&  2^4{:}(S_3 \wr S_2) & 2,4 & H<2^4{:}S_6 \\
{\rm M}_{23} & 2^4{:}(A_4 \times 3).2 & & H < 2^4{:}(A_5 \times 3){:}2,\, H< 2^4{:}A_7 \\
{\rm M}_{24} & 2^6{:}(S_4 \times S_3) & & H < 2^6{:}({\rm L}_{3}(2) \times S_3) \\
& 2^6{:}3.(S_3 \wr S_2),\, 2^6{:}3.S_4 & & H < 2^6{:}3.S_6 \\
& 2^6{:}(S_4 \times S_3) & 2 & H<2^6{:}3.S_6 \\
& 2^4{:}(A_4 \wr S_2).2 & 2 & H<2^4{:}A_8 \\
{\rm J}_2 & 2^{2+4}{:}(3 \times S_3) & 2,3 & \mbox{$H$ maximal in $G$} \\
{\rm J}_{2}.2 & 2^{2+4}.(S_3 \times S_3) & 2,3,6 & \mbox{$H$ maximal in $G$} \\ 
& 2^{1+4}.S_4 & 2 & H < 2^{1+4}.S_5 \\
{\rm Co}_2 & 2^{4+10}.(S_4 \times S_3) & & H <  2^{4+10}.(S_5 \times S_3),\, H < 2^{1+8}{:}{\rm Sp}_6(2)  \\
{\rm Fi}_{22} & 3^{1+6}{:}2^{3+4}{:}3^2{:}2 & & \mbox{$H$ maximal in $G$} \\ 
{\rm Fi}_{22}.2 & 3^{1+6}{:}2^{3+4}{:}3^2{:}2.2 & 2 & \mbox{$H$ maximal in $G$} \\ 
& (2 \times 2^{1+8}{:}3^3.S_4){:}2 &  & H <  (2 \times 2^{1+8}{:}{\rm U}_4(2){:}2){:}2 \\ \hline
\end{array}
\]
\caption{The soluble subgroups $H<G$ with $b(G,H) = 3$, $G_0 \in \mathcal{A}$}
\label{tab:solH}
\end{table}

\begin{rem}\label{r:sol2}
Let $G = {\rm Fi}_{23}$ and let $H$ be a soluble subgroup with $b(G,H) = 3$. By extending the analysis in the proof of Theorem \ref{t:all}, using the same methods, one can check that either
\begin{itemize}\addtolength{\itemsep}{0.2\baselineskip}
\item[{\rm (a)}] $H$ has index at most $8$ in a soluble maximal subgroup $M = 3^{1+8}.2^{1+6}.3^{1+2}.2S_4$; or
\item[{\rm (b)}] $H$ has index at most $2$ in a maximal parabolic subgroup $K = P_2$ of the maximal subgroup $M = {\rm P\O}_8^{+}(3){:}S_3$ of $G$.
\end{itemize}
In (a), we find that there is a subgroup $H$ of $M$ with $b(G,H) = 3$ for each possible index $|M:H| \in \{1,2,3,4,6,8\}$. In (b), $K = P_2$ is the stabiliser in $M$ of a $2$-dimensional totally singular subspace of the natural $8$-dimensional module for the socle of $M$. Note that $K$ is soluble. This case is more difficult to study computationally because $|K|^2<|G|$ and the index $|G:K| = 5009804800$ is large. In particular, we cannot use \texttt{DoubleCosetRepresentatives} to determine if $K$ has a regular orbit on $G/K$. However, we can use a computational technique from \cite{BOW} to show that $b(G,K) = 3$, which avoids the problem of determining all of the $(K,K)$ double cosets in $G$. The idea is to search randomly for a set $T$ of elements in $G$ such that all the double cosets $KgK$ with $g \in T$ are distinct and the following two conditions are satisfied:
\begin{itemize}\addtolength{\itemsep}{0.2\baselineskip}
\item[{\rm (i)}] $|KgK|<|K|^2$ for all $g \in T$; and
\item[{\rm (ii)}] $\sum_{g \in T}|KgK|>|G| - |K|^2$
\end{itemize}
If we can find a set with these properties, then this immediately rules out the existence of a regular $K$-orbit on $G/K$ and thus $b(G,K) \geqs 3$. As in \cite{BOW}, this approach can be implemented in {\sc Magma} and we can use it to show that $b(G,K) = 3$. In addition, we find that $K$ has an index-two subgroup $H$ with $b(G,H) = 3$.  
\end{rem}

\section{Proof of Theorem \ref{t:main1}, $G = \mathbb{M}$}\label{s:proof2}

To complete the proof of Theorem \ref{t:main2}, it remains to consider the groups $\mathbb{M}$ and $\mathbb{B}$. In this section, we assume $G = \mathbb{M}$ and we will prove the following stronger result.

\begin{thm}\label{t:monster}
Let $G = \mathbb{M}$ be the Monster and let $H$ be a proper subgroup of $G$. Then $b(G,H) \leqs 3$, with equality if and only if $H = 2.\mathbb{B}$.
\end{thm}

It is plain to see that the computational methods we used in the proof of Theorem \ref{t:all} are not applicable here. For example, the minimal degree of a faithful permutation representation of $G$ is $97239461142009186000$, while the dimension of a faithful linear representation over any field is at least $196882$. To proceed, we will bound $b(G,H)$ by applying a probabilistic approach based on fixed point ratio estimates, which is a powerful method introduced by Liebeck and Shalev in \cite{LSh2}. We will use similar methods to handle the Baby Monster in the next section and we recall the general set-up here.

Let $G \leqs {\rm Sym}(\O)$ be a transitive permutation group on a finite set $\O$ with point stabiliser $H$. Given a positive integer $c$, let $Q(G,H,c)$ be the probability that a randomly chosen $c$-tuple of points in $\O$ does not form a base for $G$, so $b(G,H) \leqs c$ if and only if $Q(G,H,c)<1$. Notice that a subset $\{\a_1, \ldots, \a_c \} \subseteq \O$ is not a base for $G$ if and only if there exists a prime order element $x \in G$ fixing each $\a_i$. Now the probability that $x$ fixes a randomly chosen element of $\O$ is given by the \emph{fixed point ratio}
\[
{\rm fpr}(x,G/H) = \frac{|C_{\O}(x)|}{|\O|} = \frac{|x^G \cap H|}{|x^G|},
\]
where $C_{\O}(x)$ is the set of fixed points of $x$ on $\O$, whence    
\begin{equation}\label{e:what}
Q(G,H,c) \leqs \sum_{x \in \mathcal{P}} {\rm fpr}(x,G/H)^c = \sum_{i=1}^{k} |x_i^{G}| \cdot {\rm fpr}(x_i,G/H)^c =:  \what{Q}(G,H,c),
\end{equation}
where $\mathcal{P} = \bigcup_ix_i^G$ is the set of elements of prime order in $G$. In particular, if $\what{Q}(G,H,c)<1$ then $b(G,H) \leqs c$.

The following result (\cite[Lemma 2.1]{Bur7}) provides a useful tool for bounding $\what{Q}(G,H,c)$ from above. For example, if $r$ is a prime then the lemma immediately implies that the contribution to $\what{Q}(G,H,c)$ from elements of order $r$ is at most $b(a/b)^c$, where $a = i_r(H)$ is the total number of elements of order $r$ in $H$ and $b$ is the minimal size of a conjugacy class in $G$ containing elements of order $r$.

\begin{lem}\label{l:calc}
Suppose $x_{1}, \ldots, x_{m}$ represent distinct $G$-classes such that $\sum_{i}{|x_{i}^{G}\cap H|}\leqs a$ and $|x_{i}^{G}|\geqs b$ for all $i$. Then 
\[
\sum_{i=1}^{m} |x_i^{G}| \cdot {\rm fpr}(x_i,G/H)^c \leqs b(a/b)^c
\]
for every positive integer $c$.
\end{lem}

We are now ready to prove Theorem \ref{t:monster}.

\begin{proof}[Proof of Theorem \ref{t:monster}]
In view of \cite[Theorem 1]{BOW}, it suffices to show that $b(G,K)=2$ for every maximal subgroup $K$ of $L = 2.\mathbb{B}$, where $L$ is the centraliser of a \texttt{2A} involution in $G$. Let $Z = \la z \ra$ be the centre of $L$, so $Z \leqs K$ and $\bar{K} = K/Z$ is a maximal subgroup of $\bar{L} = L/Z = \mathbb{B}$. The possibilities for $\bar{K}$ are known up to conjugacy in $\bar{L}$ and the relevant character tables are available in the \textsf{GAP} Character Table Library \cite{GAPCTL}. We will use the character tables, together with the corresponding fusion maps from $\bar{K}$-classes to $\bar{L}$-classes, to establish the bound $\what{Q}(G,K,2)<1$, which gives $b(G,K) = 2$. It will be convenient to write
\[
\what{Q}(G,K,2) = \a+\b,
\]
where $\a$ and $\b$ denote the contributions from elements of odd prime order and involutions, respectively.

\vs

\noindent \emph{Claim 1. We have $\a<2^{-6}$.}

\vs 

Let $\pi$ be the set of odd prime divisors of $|\bar{K}|$ and set $a_r = i_r(K) = i_r(\bar{K})$ for each $r \in \pi$, which is easily calculated from the character table of $\bar{K}$. If we take $b_r$ to be the minimal size of a $G$-class of elements of order $r$, then Lemma \ref{l:calc} implies that 
$\a \leqs  \sum_{r \in \pi}a_r^2/b_r$ and it is routine to check that this upper bound is less than $2^{-6}$ for $\bar{K} \ne 2.{}^2E_6(2){:}2$. 

Finally, suppose $\bar{K} = 2.{}^2E_6(2){:}2$. Let $\pi'$ be the set of primes $r \geqs 5$ dividing $|\bar{K}|$ and define $a_r$ and $b_r$ as above for $r \in \pi'$. A more careful calculation is required for elements of order $3$. First note that there are $3$ classes of such elements in $G$, labelled \texttt{3A}, \texttt{3B} and \texttt{3C} in \cite{Atlas}. Similarly, there are $3$ classes in $\bar{K}$, which we will denote by $\texttt{3A}''$, $\texttt{3B}''$ and $\texttt{3C}''$, and $2$ classes in $\bar{L}$ and $L$, denoted $\texttt{3A}'$ and $\texttt{3B}'$. By considering the relevant fusion maps in \cite{GAPCTL}, we deduce that 
\begin{align*}
\texttt{3A}'',\; \texttt{3B}'' & \mapsto \texttt{3A}' \mapsto \texttt{3A} \\
\texttt{3C}'' & \mapsto \texttt{3B}' \mapsto \texttt{3B}
\end{align*}
Therefore, the contribution to $\what{Q}(G,K,2)$ from elements of order $3$ is precisely $c_1^2/d_1+c_2^2/d_2$, where
\begin{align*}
c_1 & = |\texttt{3A} \cap K| = |\texttt{3A}''|+ |\texttt{3B}''| = 2773871493120 + 48820138278912 \\
c_2 & = |\texttt{3B} \cap K| = |\texttt{3C}''| = 7594243732275200 
\end{align*}
and $d_1 = |\texttt{3A}|$, $d_2 = |\texttt{3B}|$. We conclude that
\[
\a < \sum_{r \in \pi'}a_r^2/b_r + c_1^2/d_1+c_2^2/d_2 < 2^{-6}
\]
and the proof of Claim 1 is complete.

\vs

\noindent \emph{Claim 2. We have $\b<2^{-2}$.}

\vs 

Now let us turn to involutions. First note that $G$ has $2$ classes of involutions, labelled \texttt{2A} and \texttt{2B}. Similarly, there are $5$ and $4$ such classes in $L = 2.\mathbb{B}$ and $\bar{L} = \mathbb{B}$, respectively, and we will use the labelling of these classes given in Table \ref{tab:binv}. The fusion map from $L$-classes to $G$-classes gives
\begin{align*}
\texttt{2A}'',\; \texttt{2B}'', \; \texttt{2C}'' & \mapsto \texttt{2A} \\
\texttt{2D}'',\; \texttt{2E}'' & \mapsto \texttt{2B}
\end{align*}
Let $\bar{x} = Zx \in \bar{L}$ be an involution. If $\bar{x} \in \texttt{2A}'$ then $x$ and $zx$ are $L$-conjugate involutions and $x^L = \texttt{2B}''$. Similarly, if $\bar{x} \in \texttt{2D}'$ then $x^L = (zx)^L = \texttt{2E}''$. For $\bar{x} \in \texttt{2B}'$ we observe that $x$ and $xz$ are non-conjugate involutions, representing the classes $\texttt{2C}''$ and $\texttt{2D}''$ in $L$. Finally, we note that the involutions in $\texttt{2C}'$ lift to elements of order $4$ in $L$.

\renewcommand{\arraystretch}{1.2}
\begin{table}
\[
\begin{array}{ll} \hline
L & \bar{L} \\ \hline
\texttt{2A}'' \;\; 1 & \texttt{2A}' \;\; 13571955000 \\
\texttt{2B}'' \;\; 27143910000 & \texttt{2B}' \;\; 11707448673375 \\
\texttt{2C}'' \;\; 11707448673375 & \texttt{2C}' \;\; 156849238149120000 \\
\texttt{2D}'' \;\; 11707448673375 & \texttt{2D}' \;\; 355438141723665000 \\
\texttt{2E}'' \;\; 710876283447330000 &  \\ \hline
\end{array}
\]
\caption{The involutions in $L=2.\mathbb{B}$ and $\bar{L} = \mathbb{B}$}
\label{tab:binv}
\end{table}
\renewcommand{\arraystretch}{1}

In view of these observations, it follows that if 
\[
|\texttt{2A}' \cap \bar{K}| = a, \; |\texttt{2B}' \cap \bar{K}| = b,\; |\texttt{2D}' \cap \bar{K}| = c,
\]
then
\begin{equation}\label{e:abc}
|\texttt{2A} \cap K| = 1+2a+b,\;\; |\texttt{2B} \cap K| = b+2c.
\end{equation}

Suppose $\bar{K} = 2.^2E_6(2){:}2$. Working with the fusion map from $\bar{K}$-classes to $\bar{L}$-classes, we calculate that 
\begin{align*}
a & = 1+3968055+23113728 \\
b & = 3968055+3142699560 \\
c & = 3142699560+2639867630400+1609062174720
\end{align*}
and in view of \eqref{e:abc} we deduce that $\b<2^{-2}$. An entirely similar argument applies when $\bar{K} = 2^{1+22}.{\rm Co}_2$, ${\rm Fi}_{23}$ and $2^{9+16}.{\rm Sp}_8(2)$. In the remaining cases, we only need to calculate $i_2(\bar{K})$ since $i_2(K) \leqs 2i_2(\bar{K})+1=d$ and one checks that $\b \leqs d^2/e<2^{-2}$ with $e = |\texttt{2A}|$.

\vs

This establishes Claim 2 and we conclude that $\what{Q}(G,K,2) < 2^{-6}+2^{-2}<1$, which gives $b(G,K) = 2$ as required.
\end{proof}

\section{Proof of Theorem \ref{t:main1}, $G = \mathbb{B}$}\label{s:proof3}

In this final section we complete the proof of Theorem \ref{t:main2} by handling the case where $G$ is the Baby Monster $\mathbb{B}$. Our main result is the following.

\begin{thm}\label{t:bmonster}
Let $G = \mathbb{B}$ be the Baby Monster and let $H$ be a soluble subgroup of $G$. Then $b(G,H) \leqs 3$.
\end{thm}

Let $H$ be a soluble subgroup of $G$. By applying the main theorem of \cite{BOW}, we may assume that $H <L<G$, where $L = 2.{}^2E_6(2){:}2$ is the centralizer of a \texttt{2A} involution in $G$. In particular, it  suffices to show that $b(G,M) \leqs 3$ for all $M \in \mathcal{M}$, where $\mathcal{M}$ is the union of the maximal subgroups of $2.{}^2E_6(2)$ and $L$, excluding $2.{}^2E_6(2)$ itself (note that if $M = 2.{}^2E_6(2)$ then $|G|>|G:M|^3$ and thus $b(G,M) = 4$ since $b(G,L) = 4$). We begin by recording some preliminary observations.

\subsection{Preliminaries}\label{ss:prel2}

Let $M \ne 2.{}^2E_6(2)$ be a maximal subgroup of $L$ and let $Z = \la z \ra$ be the centre of $L$. Then $Z \leqs M$ and $M/Z$ is a maximal subgroup of the almost simple group $L/Z = {}^2E_6(2){:}2$. For the remainder, we will write $\bar{M} = M/Z$, $\bar{L} = L/Z$ and $\bar{x} = Zx$ for $x \in L$. We need to recall some facts on the subgroup structure of $\bar{L}$ and the conjugacy classes of involutions in $L$, $\bar{L}$ and $G$.

Write $\bar{L}' = {}^2E_6(2) = (X_{\s})'$, where $X = E_6$ is a simple algebraic group of adjoint type over an algebraically closed field of characteristic $2$ and $\s$ is a Steinberg endomorphism of $X$ such that 
\[
X_{\s} = \{ x \in X \,:\, x^{\sigma} = x\} = {}^2E_6(2){:}3 = {\rm Inndiag}({}^2E_6(2)).
\]
The maximal subgroups of $\bar{L}$ and $\bar{L}'$ have been determined up to conjugacy (see \cite{Wilson}, which confirms that the list of maximal subgroups presented in the {\sc Atlas} \cite{Atlas} is complete) and the possibilities for $\bar{M}$ are as follows:
\[
\begin{array}{rl}
\mbox{Parabolic:} & P_{1,6}, \, P_2, \, P_{3,5}, \, P_4 \\
\mbox{Algebraic:} & {\rm O}_{10}^{-}(2), \, S_3 \times {\rm U}_{6}(2){:}2, \, S_3 \times \O_{8}^{+}(2){:}S_3, \, {\rm U}_{3}(8){:}6, \, ({\rm L}_{3}(2) \times {\rm L}_{3}(4){:}2).2,  \\
&  3^{1+6}{:}2^{3+6}{:}3^2{:}2^2,\, {\rm U}_{3}(2){:}2 \times G_2(2), \, F_4(2) \times 2 \\
\mbox{Almost simple:} & {\rm SO}_{7}(3), \, {\rm Fi}_{22}{:}2
\end{array}
\]
Here we adopt the standard notation for the maximal parabolic subgroups of $\bar{L}$, which corresponds to the usual labelling of the nodes of the Dynkin diagram of type $E_6$ (see Table \ref{tab:parab} for further details on the structure of these subgroups). The algebraic subgroups are of the form $N_{\bar{L}}(Y_{\s})$, where $Y$ is a positive dimensional non-parabolic $\s$-stable closed subgroup of the ambient algebraic group $X$. For example, if $\bar{M} = 3^{1+6}{:}2^{3+6}{:}3^2{:}2^2$ then the connected component of $Y$ is of type $A_2^3$ and we will refer to ${\rm SU}_3(2)^3$ as the \emph{type} of $\bar{M}$. We can extend this usage of type to the other algebraic subgroups to provide an approximate description of the given subgroup's structure.

\begin{rem}\label{r:red}
From the description of the maximal subgroups of $\bar{L}' = {}^2E_6(2)$ in \cite{Atlas,Wilson}, we immediately deduce that in order to prove Theorem \ref{t:bmonster} it suffices to show that $b(G,M) \leqs 3$ for every maximal subgroup $M \ne 2.{}^2E_6(2)$ of $L$. 
\end{rem}

Next we recall some information on the conjugacy classes of involutions in $L$ and $G$. Here it is convenient to observe that the character tables of $L$ and $G$ are available in the \textsf{GAP} Character Table Library \cite{GAPCTL}, together with the fusion map from $L$-classes to $G$-classes. Following \cite{Atlas}, we use the labels \texttt{2A},  \texttt{2B},  \texttt{2C} and \texttt{2D} for the $4$ classes of involutions in $G$. Similarly, there are $5$ classes of involutions in $\bar{L} = {}^2E_6(2){:}2$; $3$ are contained in ${}^2E_6(2)$ (these unipotent classes are labelled $A_1, A_1^2$ and $A_1^3$ in \cite[Table 22.2.3]{LS_book}) and the other $2$ classes comprise involutory graph automorphisms of ${}^2E_6(2)$. We will use the labelling 
$\texttt{2A}', \ldots, \texttt{2E}'$ for these classes, noting that the $\texttt{2A}'$ involutions are long root elements and $C_{{}^2E_6(2)}(\tau) = F_4(2)$ for the graph automorphisms $\tau$ in \texttt{2D}$'$. Finally, we observe that each involution in $\bar{L}$ is of the form $\bar{x}=Zx$ for some involution $x \in L$ (in other words, each involution in $\bar{L}$ lifts to an involution in $L$) and we note that $|x^L| = |\bar{x}^{\bar{L}}|$, unless $\bar{x}$ is an involution in the class \texttt{2C}$'$ in which case $|x^L| = 2|\bar{x}^{\bar{L}}|$ (that is, $x$ and $xz$ are $L$-conjugate). Including the central involution $z$, it follows that there are $10$ classes of involutions in $L$, which we will denote by the labels $\texttt{2A}'', \ldots, \texttt{2J}''$. The size of each class of involutions in $L$, $\bar{L}$ and $G$ is recorded in Table \ref{tab:inv}. From the stored fusion map in \cite{GAPCTL}, we observe that the involutions in $L$ embed in $G$ as follows:
\begin{align}\label{e:fus}
\begin{split}
\texttt{2A}'', \; \texttt{2C}'', \; \texttt{2G}'' & \mapsto \; \texttt{2A} \\
\texttt{2B}'', \; \texttt{2D}'' & \mapsto \; \texttt{2B} \\
\texttt{2H}'', \; \texttt{2I}'' & \mapsto \; \texttt{2C} \\
\texttt{2E}'', \; \texttt{2F}'', \; \texttt{2J}'' & \mapsto \; \texttt{2D} 
\end{split}
\end{align}
In particular, if $M$ is a maximal subgroup of $L$ and
\[
|\texttt{2A}' \cap \bar{M}| = a,\; |\texttt{2B}' \cap \bar{M}| = b, \; |\texttt{2C}' \cap \bar{M}| = c,\; |\texttt{2D}' \cap \bar{M}| = d, \; |\texttt{2E}' \cap \bar{M}| = e,
\]
then
\[
|\texttt{2A} \cap M| = 1+ a+ d, \; |\texttt{2B} \cap M| = a+ b, \; |\texttt{2C} \cap M| = d+e,\; 
|\texttt{2D} \cap M| = b+2c+e
\]
and $i_2(M) = 2i_2(\bar{M})+1$.

\renewcommand{\arraystretch}{1.2}
\begin{table}
\[
\begin{array}{lll} \hline
L & \bar{L} & G \\ \hline
\texttt{2A}'' \;\; 1 & \texttt{2A}' \;\; 3968055 & \texttt{2A}\;\; 13571955000 \\
\texttt{2B}'' \;\; 3968055 & \texttt{2B}' \;\; 3142699560 & \texttt{2B}\;\; 11707448673375 \\
\texttt{2C}'' \;\; 3968055 & \texttt{2C}' \;\; 1319933815200 & \texttt{2C}\;\; 156849238149120000 \\
\texttt{2D}'' \;\; 3142699560 & \texttt{2D}' \;\; 23113728 & \texttt{2D}\;\; 355438141723665000 \\
\texttt{2E}'' \;\; 3142699560 & \texttt{2E}' \;\; 1609062174720 &  \\
\texttt{2F}'' \;\; 2639867630400 &  &  \\
\texttt{2G}'' \;\; 23113728 &  &  \\
\texttt{2H}'' \;\; 23113728 &  &  \\
\texttt{2I}'' \;\; 1609062174720 &  &  \\
\texttt{2J}'' \;\; 1609062174720 &  &  \\ \hline
\end{array}
\]
\caption{The involutions in $L$, $\bar{L}$ and $G$}
\label{tab:inv}
\end{table}
\renewcommand{\arraystretch}{1}

For $r \in \{3,5\}$, we note that $G$ contains two classes of elements of order $r$, labelled 
\texttt{3A}, \texttt{3B} and \texttt{5A}, \texttt{5B} respectively. In addition, if $x \in G$ has odd prime order then $|x^G| \geqs |\texttt{3A}|$. Similarly, $|x^G| \geqs |\texttt{5A}|$ for all $x \in G$ of prime order at least $5$.

Finally, the following elementary observation will be useful.

\begin{lem}\label{l:easy}
If $M$ is a maximal subgroup of $L$ and $b(\bar{L},\bar{M}) = 2$, then $b(G,M) \leqs 3$.
\end{lem}

\begin{proof}
By definition, there exists $\bar{x} \in \bar{L}$ such that $(M \cap M^x)/Z = \bar{M} \cap \bar{M}^{\bar{x}} = 1$, whence $M \cap M^x = Z$. Since $L$ is a core-free subgroup of $G$, there exists $y \in G$ such that $Z \cap L^y = 1$ and we conclude that $M \cap M^x \cap M^y = 1$.
\end{proof} 

\subsection{Non-parabolic subgroups}\label{ss:nonpar}

\begin{lem}\label{l:one}
We have $b(G,M) \leqs 3$ if $\bar{M}$ is of type ${\rm SU}_{3}(2)^3$ or ${\rm SL}_{3}(2) \times {\rm SL}_{3}(4)$.
\end{lem}

\begin{proof}
In both cases, \cite[Proposition 4.2]{BTh} gives $b(\bar{L},\bar{M}) = 2$ and so the result follows from Lemma \ref{l:easy}.
\end{proof}
 
\begin{lem}\label{l:two}
We have $b(G,M) \leqs 3$ if $\bar{M}$ is of type ${\rm SU}_{3}(2) \times G_2(2)$, ${\rm SU}_{3}(8)$ or $3 \times \O_8^{+}(2)$.
\end{lem}

\begin{proof}
We will show that $\what{Q}(G,M,3)<1$ (see \eqref{e:what}), which gives $b(G,M) \leqs 3$.

First assume $\bar{M}$ is of type ${\rm SU}_{3}(2) \times G_2(2)$. Here $|M| = 3483648 =a_1$ and $|x^G| \geqs |\texttt{2A}| = b_1$ for every element $x \in G$ of prime order. By applying Lemma \ref{l:calc}, we deduce that 
\[
\what{Q}(G,M,3) \leqs b_1(a_1/b_1)^3 < 1
\]
and the result follows.

Next suppose $\bar{M}$ is of type ${\rm SU}_{3}(8)$, so $|M| = 66189312 = a_1$ and we recall that $|x^G| \geqs |\texttt{3A}| = b_1$ for all $x \in G$ of odd prime order. It is easy to check that $i_2(\bar{M}) = 14535$ and thus $i_2(M) = 2i_2(\bar{M})+1 = 29071 = a_2$. As noted above, $|x^G| \geqs |\texttt{2A}| = b_2$ for every involution $x \in G$ and we conclude that \[
\what{Q}(G,M,3) \leqs b_1(a_1/b_1)^3 +b_2(a_2/b_2)^3< 1.
\]
A similar argument applies when $\bar{M}$ is of type $3 \times \O_8^{+}(2)$, noting that $i_2(\bar{M}) = 733503$.
\end{proof}

\begin{lem}\label{l:three}
We have $b(G,M) \leqs 3$ if $\bar{M}$ is of type ${\rm SL}_{2}(2) \times {\rm SU}_6(2)$.
\end{lem}

\begin{proof}
Here $\bar{M} = S_3 \times {\rm U}_{6}(2){:}2$ and we calculate that $i_2(\bar{M}) = 2872191$. As before, the contribution to $\what{Q}(G,M,3)$ from elements of odd prime order is less than $b_1(a_1/b_1)^3$, where $a_1 = |M| = 220723937280$ and $b_1 = |\texttt{3A}|$.
Similarly, the combined contribution from involutions in the classes \texttt{2B}, \texttt{2C} and \texttt{2D} is at most $b_2(a_2/b_2)^3$, where $a_2 = 2{\cdot}2872191+1$ and $b_2 = |\texttt{2B}|$. 
It remains to estimate the contribution to $\what{Q}(G,M,3)$ from the involutions in \texttt{2A}.

Let $x \in G$ be an involution in the class \texttt{2A}. As noted in Section \ref{ss:prel2} (see \eqref{e:fus}), we have $x^G \cap L = x_1^L \cup x_2^L \cup x_3^L$, where $x_1 = z$ is the central involution, $\bar{x}_2 \in \bar{M}$ is a long root element in ${}^2E_6(2)$ and $\bar{x}_3$ is a graph automorphism of ${}^2E_6(2)$ with centraliser $F_4(2)$. By \cite[Proposition 1.13]{LLS1}, the long root elements in $\bar{M}$ correspond to the long root elements in the factors ${\rm SL}_2(2)$ and ${\rm U}_{6}(2)$ and thus
\[
|x_2^L \cap M| = |\texttt{2A}'\cap \bar{M}| = 3+\frac{|{\rm GU}_{6}(2)|}{2^9|{\rm GU}_4(2)||{\rm GU}_1(2)|} = 696.
\]
Similarly, by appealing to the proof of \cite[Lemma 4.16]{BLS} we calculate that
\[
|x_3^L \cap M| = |\texttt{2D}' \cap \bar{M}| = \frac{|{\rm U}_6(2)|}{|{\rm Sp}_6(2)|} = 6336.
\]
Therefore, if we set $a_3 = 1+696+6336 = 7033$ and $b_3 = |\texttt{2A}|$, then $b_3(a_3/b_3)^3$ is the contribution to $\what{Q}(G,M,3)$ from \texttt{2A} involutions. We conclude that
\[
\what{Q}(G,M,3) \leqs \sum_{i=1}^{3}b_i(a_i/b_i)^3 < 1
\]
and thus $b(G,M) \leqs 3$ as required. 
\end{proof}

\begin{lem}\label{l:four}
We have $b(G,M) \leqs 3$ if $\bar{M} = {\rm O}_{10}^{-}(2)$, ${\rm SO}_7(3)$ or ${\rm Fi}_{22}{:}2$.
\end{lem}

\begin{proof}
First assume $\bar{M} = {\rm O}_{10}^{-}(2)$. The character tables of $\bar{M}$ and $\bar{L}$ are available in \cite{GAPCTL} and we can use the \textsf{GAP} function \texttt{PossibleClassFusions} to determine the set of possible fusion maps from the set of conjugacy classes in $\bar{M}$ to the set of classes in $\bar{L}$. We find that there are only two such maps, both of which give the same values for $|\bar{x}^{\bar{L}} \cap \bar{M}|$ with $\bar{x} \in \bar{L}$ an involution. For example, we find that 
\[
|\texttt{2A}' \cap \bar{M}| = 19635,\;\; |\texttt{2B}' \cap \bar{M}| = 67320+706860 = 774180
\]
and thus \eqref{e:fus} gives
\[
|\texttt{2B} \cap M| = |\texttt{2B}'' \cap M| + |\texttt{2D}'' \cap M| = 19635+774180 = 793815.
\]
We record $|x^G \cap M|$ in Table \ref{tab:inv2} for each involution $x \in G$ and it is easy to check that the contribution to $\what{Q}(G,M,3)$ from involutions is less than $2^{-24}$. Finally, we calculate that $\bar{M}$ (and thus $M$ also) contains precisely $a =4547907351296$ elements of odd prime order and we recall that $|x^G| \geqs |\texttt{3A}| = b$ for all such elements in $G$. Therefore, 
\[
\what{Q}(G,M,3) < 2^{-24} + b(a/b)^3 < 1
\]
as required.

An entirely similar argument applies in the two remaining cases and we omit the details. Note that if $\bar{M} = {\rm Fi}_{22}{:}2$ then the fusion map for the embedding in $\bar{L}$ is available in \cite{GAPCTL}, while the function \texttt{PossibleClassFusions} returns a unique map when $\bar{M} = {\rm SO}_7(3)$. In both cases, $|x^G \cap M|$ is recorded in Table \ref{tab:inv2} for each involution $x \in G$.  
\end{proof}

\renewcommand{\arraystretch}{1.2}
\begin{table}
\[
\begin{array}{lcccc} \hline
\bar{M} & |\texttt{2A} \cap M| &  |\texttt{2B} \cap M| &|\texttt{2C} \cap M| &|\texttt{2D} \cap M| \\ \hline
P_{1,6} & 138296 & 4871775 & 355401728 & 1119260008 \\
P_2 & 451640 & 16107103 & 927793152 & 2597999976 \\
P_{3,5} & 31800 & 595551 & 31129600 &  101543272 \\
P_4 & 66616 & 1443423 & 46202880 & 166458728 \\
{\rm O}_{10}^{-}(2) & 20164 & 793815 & 36757748 &  79943000 \\
{\rm SO}_7(3) & 730 & 22464 & 309960 & 995085 \\
F_4(2) \times 2 & 139231 & 4524975 & 355384575 & 1061489520 \\
{\rm Fi}_{22}{:}2 & 65287 & 1219725 & 41760576 & 115887915 \\ \hline
\end{array}
\]
\caption{Involutions in some maximal subgroups $M$ of $L$}
\label{tab:inv2}
\end{table}
\renewcommand{\arraystretch}{1}

\begin{lem}\label{l:five}
We have $b(G,M) \leqs 3$ if $\bar{M} = F_4(2) \times 2$.
\end{lem}

\begin{proof}
First observe that the character tables of $F_4(2)$ and ${}^2E_6(2)$ are available in \cite{GAPCTL}, together with the corresponding fusion map on conjugacy classes. If $x \in G$ has order $3$ then $|x^G| \geqs |\texttt{3A}| = b_1$ and we calculate that $i_3(M) = i_3(\bar{M}) = 72489697280 = a_1$. Similarly, if $x \in G$ has prime order $r \geqs 5$ then $|x^G| \geqs |\texttt{5A}| = b_2$ and we note that $\bar{M}$ (and thus $M$) contains precisely $a_2 = 650797277773824$ such elements.

Now assume $x \in G$ is an involution. By working with the fusion map from classes in $F_4(2)$ to classes in ${}^2E_6(2)$, we calculate 
\[
|\texttt{2A}' \cap \bar{M}| = 69615,\;\; |\texttt{2B}' \cap \bar{M}| = 69615+4385745 = 4455360, \;\; |\texttt{2C}' \cap \bar{M}|= 350859600,
\]
while the proof of \cite[Lemma 5.4]{LLS2} gives
\begin{align*}
|\texttt{2D}' \cap \bar{M}| & =\frac{|F_4(2)|}{2^{15}|{\rm Sp}_6(2)|} = 69615  \\
|\texttt{2E}' \cap \bar{M}| & =\frac{|F_4(2)|}{2^{15}|{\rm Sp}_6(2)|}+\frac{|F_4(2)|}{2^{24}(2^2-1)(2^4-1)}+\frac{|F_4(2)|}{2^{20}(2^2-1)^2} = 355314960.
\end{align*}
We can now calculate $|x^G \cap M|$ for each involution $x \in G$ (see Table \ref{tab:inv2}) and we find that the contribution to $\what{Q}(G,M,3)$ from involutions is less than $2^{-15}$.

Therefore, bringing the above bounds together, we conclude that
\[
\what{Q}(G,M,3) < 2^{-15} + b_1(a_1/b_1)^3 + b_2(a_2/b_2)^3 < 1
\]  
and the result follows.
\end{proof}

\subsection{Parabolic subgroups}\label{ss:par}

To complete the proof of Theorem \ref{t:bmonster}, we may assume $\bar{M}$ is a maximal parabolic subgroup of $\bar{L}$. As previously discussed, there are $4$ conjugacy classes of such subgroups, with representatives labelled $P_{1,6}$, $P_2$, $P_{3,5}$ and $P_4$ with respect to the usual numbering of the nodes of the Dynkin diagram of type $E_6$. The structure and index of each maximal parabolic subgroup is presented in Table \ref{tab:parab} (the given values for $|\texttt{3A}\cap M|$ and $|\texttt{3B} \cap M|$ will be explained in the proof of Lemma \ref{l:order3} below).

\renewcommand{\arraystretch}{1.2}
\begin{table}
\[
\begin{array}{lcccc} \hline
\bar{M} & \mbox{Structure} & |\bar{L}:\bar{M}| & |\texttt{3A} \cap M| &  |\texttt{3B} \cap M| \\ \hline
P_{1,6} & 2^{8+16}{:}{\rm O}_8^{-}(2) & 23108085 & 3315597312 & 79859548160 \\
P_2 & 2^{1+20}{:}{\rm U}_6(2){:}2 & 3968055 & 58617495552 & 51673825280  \\
P_{3,5}   & 2^{3+4+12+12}{:}(A_5 \times {\rm L}_3(2) \times 2)  & 3535537005  & 266338304 & 4697620480 \\
P_4 & 2^{2+9+18}{:}({\rm L}_3(4){:}2 \times S_3) & 1178512335 & 1763704832 & 4697620480 \\ \hline
\end{array}
\]
\caption{The maximal parabolic subgroups of $\bar{L}$}
\label{tab:parab}
\end{table}
\renewcommand{\arraystretch}{1}

Write 
\[
\what{Q}(G,M,3) = \a+\b+\gamma,
\]
where $\a$, $\b$ and $\gamma$ denote the contributions from elements of order at least $5$, elements of order $3$ and involutions, respectively. We will estimate each of these contributions in turn.

\begin{lem}\label{l:order5}
We have $\a<2/3$.
\end{lem}

\begin{proof}
Clearly, $M$ contains fewer than 
\[
|\bar{M}| \leqs |P_2| = 2^{22}|{\rm U}_6(2)| = 38574303876218880 = a
\]
elements of prime order $r \geqs 5$ and we recall that $|x^G| \geqs |\texttt{5A}|=b$ for every such element. Therefore, $\a \leqs b(a/b)^3 < 2/3$. 
\end{proof}

Let us now consider $\b$ and $\gamma$; we need to show that $\beta+\gamma<1/3$. To do this, we will use techniques from \cite{LLS2} to evaluate the corresponding permutation character $1^{\bar{L}}_{\bar{M}}$, noting that
\begin{equation}\label{e:perm}
|\bar{x}^{\bar{L}} \cap \bar{M}| = \frac{|\bar{x}^{\bar{L}}|}{|\bar{L}:\bar{M}|} \cdot  1^{\bar{L}}_{\bar{M}}(\bar{x})
\end{equation}
for all $\bar{x} \in \bar{L}$.

\begin{lem}\label{l:order3}
We have $\b<2^{-19}$.
\end{lem}

\begin{proof}
First recall that there are $2$ classes of elements of order $3$ in $G$, labelled \texttt{3A} and \texttt{3B}. In addition, there are $3$ such classes in both $L$ and $\bar{L}$, which we will label $\texttt{3A}''$, $\texttt{3B}''$ and $\texttt{3C}''$, and we will view them as conjugacy classes in $L$. Here
\begin{align*}
|\texttt{3A}''| & = \frac{|{}^2E_6(2)|}{(2+1)|{\rm U}_6(2)|} = 2773871493120 \\
|\texttt{3B}''| & = \frac{|{}^2E_6(2)|}{(2+1)^2|\Omega^{+}_8(2)|} = 48820138278912 \\
|\texttt{3C}''| & = \frac{|{}^2E_6(2)|}{|{\rm SU}_3(2)|^3} = 7594243732275200
\end{align*}
The fusion map in \cite{GAPCTL} indicates that these conjugacy classes embed in $G$ as follows: 
\begin{align}\label{e:fus2}
\begin{split}
\texttt{3A}'', \; \texttt{3B}'' & \mapsto \; \texttt{3A} \\
\texttt{3C}'' & \mapsto \; \texttt{3B} 
\end{split}
\end{align}

As before, write ${}^2E_6(2) = (X_{\s})'$, where $X= E_6$ is a simple algebraic group of adjoint type defined over an algebraically closed field of characteristic $2$ and $\s$ is a Steinberg endomorphism of $X$. Let $W  = {\rm U}_4(2).2$ be the Weyl group of $X$ and fix a set $\Pi = \{\a_1, \ldots, \a_6\}$ of simple roots for $X$. Let $\a_0$ be the highest root in the root system of $X$. As explained in \cite{FJ}, the semisimple classes in $X_{\s}$ are parameterised by pairs $(J,[w])$, where $J$ is a proper subset of $\Pi \cup \{\a_0\}$ (determined up to $W$-conjugacy), $W_J$ is the subgroup of $W$ generated by the reflections in the roots in $J$, and $[w] = W_Jw$ is a conjugacy class representative in $N_W(W_J)/W_J$. For the elements of order $3$ in ${}^2E_6(2)$ that we are interested in, we observe that 
\[
\texttt{3A}'' \longleftrightarrow (A_5T_1, [1]), \; \texttt{3B}'' \longleftrightarrow (D_4T_2, [1]), \; \texttt{3C}'' \longleftrightarrow (A_2^3, [1])
\]
under this correspondence.

With the aid of {\sc Magma}, we can evaluate the expression for $1_{\bar{M}}^{\bar{L}}(\bar{x})$ given in \cite[Corollary 3.2]{LLS2} and this allows us to compute $|\bar{x}^{\bar{L}} \cap \bar{M}|$ via \eqref{e:perm}. For example, suppose $\bar{M} = P_{1,6}$. Working over an arbitrary finite field $\mathbb{F}_q$ of characteristic $p \ne 3$, we calculate that any semisimple element in the ${\rm Inndiag}({}^2E_6(q))$-class labelled by the pair $(A_5T_1, [1])$ has exactly 
\[
q^{12} + q^{10} + q^9 + q^8 + 2q^7 + 2q^5 + q^4 + q^3 + q^2 + 1
\]
fixed points on the set of cosets of a $P_{1,6}$ parabolic subgroup. Setting $q=2$ gives $1_{\bar{M}}^{\bar{L}}(\bar{x}) = 6237$ and thus $|\texttt{3A}'' \cap \bar{M}| = 748683264$. Similarly, 
\[
|\texttt{3B}'' \cap \bar{M}| = 2566914048,\;\; |\texttt{3C}'' \cap \bar{M}| = 79859548160
\]
and in view of the fusion map in \eqref{e:fus2} we deduce that 
\[
|\texttt{3A} \cap M| = 748683264+2566914048, \;\; |\texttt{3B} \cap M| = 79859548160.
\]
The other cases are entirely similar and we record $|\texttt{3A} \cap M|$ and $|\texttt{3B} \cap M|$ in the final two columns of Table \ref{tab:parab}. It is now routine to check that the contribution to $\what{Q}(G,M,3)$ from elements of order $3$ is less than $2^{-19}$.
\end{proof}

\begin{lem}\label{l:order2}
We have $\gamma<2^{-10}$.
\end{lem}

\begin{proof}
As in the proof of the previous lemma, write ${}^2E_6(2) = (X_{\s})'$ and recall that $\bar{L} = {}^2E_6(2){:}2$ has $5$ conjugacy classes of involutions: namely, $3$ unipotent classes in ${}^2E_6(2)$, labelled $\texttt{2A}'$, $\texttt{2B}'$ and $\texttt{2C}'$ above (these are the classes 
$A_1$, $A_1^2$ and $A_1^3$ in the usual Bala-Carter notation), while the two remaining classes (labelled $\texttt{2D}'$ and $\texttt{2E}'$) comprise involutory graph automorphisms. More precisely, $C_{{}^2E_6(2)}(\bar{x}) = F_4(2)$ if $\bar{x} \in \texttt{2D}'$ and $C_{{}^2E_6(2)}(\bar{x}) = C_{F_4(2)}(t)$ if $\bar{x} \in \texttt{2E}'$, where $t \in F_4(2)$ is a long root element. Set $\chi = 1_{\bar{M}}^{\bar{L}}$.

First consider the involutions in ${}^2E_6(2)$. Here we can use the method introduced in \cite[Section 2]{LLS2} to calculate $\chi(\bar{x})$ for each involution $\bar{x} \in {}^2E_6(2)$. To do this, we first express $\chi$ as a sum of almost characters of $X_{\s}$ of the form $R_{\phi}$, where $\phi$ is a complex irreducible character of the Weyl group $W$ of $X$. We refer the reader to \cite[p.125]{BLS} for the relevant decompositions, where the labelling of irreducible characters given in \cite[Section 13.2]{Carter} is adopted. The restriction of $R_{\phi}$ to unipotent elements gives the Green functions of $X_{\s}$, which have been computed by L\"{u}beck \cite{Lub} via an algorithm due to Lusztig \cite{Lus}. This allows us to calculate $\chi(\bar{x})$ for each involution $\bar{x} \in {}^2E_6(2)$, which in turn yields $|\bar{x}^{\bar{L}} \cap \bar{M}|$ in view of \eqref{e:perm}.

For example, suppose $\bar{M} = P_{1,6}$ and $\bar{x}$ is an involution in the $\texttt{2B}'$ class. As noted in \cite[p.125]{BLS} we have
\[
\chi = R_{\phi_{1,0}}+R_{\phi_{15,5}}+R_{\phi_{20,2}}+R_{\phi_{30,3}}
\]
and working over an arbitrary finite field $\mathbb{F}_q$ of characteristic $2$ we calculate that a unipotent element in the ${\rm Inndiag}({}^2E_6(q))$-class labelled $A_1^2$ has precisely 
\[
q^{14} + q^{13} + q^{12} + 2q^{11} + q^{10} + 2q^9 + q^8 + q^7 + q^6 + q^5 + q^4 + q^2 + 1
\]
fixed points on the set of cosets of a $P_{1,6}$ parabolic subgroup.  
Setting $q=2$, this gives $\chi(\bar{x}) = 35317$ for $\bar{x} \in \texttt{2B}'$ and thus $|\texttt{2B}' \cap \bar{M}| = 4803112$. Similarly, we get 
\[
|\texttt{2A}' \cap \bar{M}| = 68663,\;\; |\texttt{2C}' \cap \bar{M}| = 379562400.
\]

A similar approach can be used to evaluate $\chi(\bar{x})$ when $\bar{x}$ is an involutory graph automorphism of ${}^2E_6(2)$ (see \cite[Proposition 2.6]{LLS2} and its proof). We thank Ross Lawther for his assistance with this calculation, which was originally used to derive the fixed point ratio bounds presented in \cite[Table 7]{BLS}. For example, for $\bar{M} = P_{1,6}$ we get 
$|\texttt{2D}' \cap \bar{M}| = 69632$ and $|\texttt{2E}' \cap \bar{M}| = 355332096$.

In this way, we can compute $|\bar{x}^{\bar{L}} \cap \bar{M}|$ for every involution $\bar{x} \in \bar{L}$. In turn, this allows us to calculate $|x^L \cap M|$ for each involution $x \in L$ and we then obtain 
$|x^G \cap M|$ via the fusion map described in \eqref{e:fus}. We record the relevant information in Table \ref{tab:inv2} and it is now a straightforward exercise to verify the bound $\gamma<2^{-10}$ for each parabolic subgroup $\bar{M}$.
\end{proof}

By combining the bounds on $\a$, $\b$ and $\gamma$ in Lemmas \ref{l:order5}-\ref{l:order2}, we conclude that 
\[
\what{Q}(G,M,3) = \a+\b+\gamma<1 
\]
and thus $b(G,M) \leqs 3$. In view of our earlier work in this section, together with Remark \ref{r:red}, this completes the proof of Theorem \ref{t:bmonster}. In particular, the proof of Theorem \ref{t:main2} is complete.

\begin{rem}\label{r:baby}
Let $G = \mathbb{B}$ be the Baby Monster and let $H$ be a soluble subgroup of $G$.  We have not been able to determine whether or not there exists an example with $b(G,H) = 3$ and we conclude by recording some further remarks on this problem.
\begin{itemize}\addtolength{\itemsep}{0.2\baselineskip}
\item[{\rm (a)}] Breuer has shown that $|H| \leqs m$, where $m = 29686813949952$ (see Chapter 6 in the manual for \cite{GAPCTL}). In particular, if $n = |G:H|$ then $\log_n|G| \leqs 1.669$. Moreover, he shows that $G$ has a unique conjugacy class of soluble subgroups of order $m$ and the maximal overgroups of such a subgroup are of the form $2^{9+16}.{\rm Sp}_{8}(2)$ and $2^{2+10+20}.({\rm M}_{22}{:}2 \times S_3)$. By \cite{BOW,NNOW}, we know that $b(G,L) = 3$ if $L$ is one of these subgroups and so further work is needed in order to determine if $b(G,H) = 2$ or $3$.

\item[{\rm (b)}] More generally, let $L$ be a maximal subgroup of $G$ containing $H$. We may assume $b(G,L) \geqs 3$, so \cite{BOW, NNOW} implies that $L$ is one of the following (up to conjugacy):
\[
2.{}^2E_6(2){:}2,\; 2^{1+22}.{\rm Co}_2,\; {\rm Fi}_{23}, \; 2^{9+16}.{\rm Sp}_8(2),\; {\rm Th},
\]
\[
(2^2 \times F_4(2)){:}2,\; 2^{2+10+20}.({\rm M}_{22}{:}2 \times S_3).
\]
Here $b(G,L) = 4$ for $L = 2.{}^2E_6(2){:}2$ and $b(G,L) = 3$ in the remaining cases.
Since $s({\rm Th})=2$ by Theorem \ref{t:main2}, we immediately deduce that $b(G,H) = 2$ if $H<{\rm Th}$. Similarly, if $H<{\rm Fi}_{23}$ then we can use the stored character tables and associated fusion maps in \cite{GAPCTL} to show that $\what{Q}(G,M,2)<1$ for every maximal subgroup $M<{\rm Fi}_{23}$, whence $b(G,H) = 2$. But this approach is not available when $H$ is contained in one of the other maximal subgroups of $G$ listed above. 

\item[{\rm (c)}] In \cite{NNOW}, the authors use the \textsf{GAP} package {\sc Orb} \cite{Orb} to prove that $b(G,L) = 3$ when $L = 2^{2+10+20}.({\rm M}_{22}{:}2 \times S_3)$ and $b(G,L) = 2$ for $L = [2^{30}].{\rm L}_{5}(2)$ (in both cases, the bound $b(G,L) \leqs 3$ was established in \cite{BOW}). However, it is not clear (to the author at least) that a similar approach could be used to compute the precise value of $s(\mathbb{B})$ and so this remains an open problem. 
\end{itemize}
\end{rem}

\end{document}